\documentclass[a4paper,12pt,reqno]{amsart}
\usepackage{mathtools}
\usepackage{amsfonts,amsmath}
\usepackage{amssymb,amsthm}		
\usepackage{amsthm}
\usepackage[T1]{fontenc}
\usepackage{lmodern}
\usepackage{hyperref}
\usepackage[all]{hypcap}
\usepackage{fullpage}
\usepackage{color}
\usepackage{graphicx}
\usepackage[dvipsnames]{xcolor}
\usepackage{fancyhdr}
\usepackage{psfrag}

\usepackage{indentfirst}
\usepackage{float}
\usepackage{latexsym}
\usepackage{graphics}
\usepackage{verbatim}
\newcommand{\ord}{{\rm ord}}

\newtheorem{theorem}{Theorem}
\newtheorem*{theorem*}{Theorem}
\newtheorem{lemma}[theorem]{Lemma}

\theoremstyle{definition}

\theoremstyle{remark}

\title{Prime powers dividing products of consecutive integer values of $x^{2^n}+1$}
\date{\today}
\author{Stephan~Baier}
\address{Stephan~Baier\\%
	Ramakrishna Mission Vivekananda Educational Research Institute\\%
	Department of Mathematics\\%
	G.\ T.\ Road, PO~Belur Math, Howrah, West Bengal~711202\\%
	India}
\email{email\_baier@yahoo.de}

\author{Pallab Kanti Dey}
\address{Pallab Kanti Dey\\%
Ramakrishna Mission Vivekananda Educational Research Institute\\%
	Department of Mathematics\\%
	G.\ T.\ Road, PO~Belur Math, Howrah, West Bengal~711202\\%
	India}
\email{pallabkantidey@gmail.com}	

\begin{document}
\maketitle
 
%\subjclass[2000]{11J71,11N36,11J25}

\maketitle

{\bf Abstract:} Let $n$ be a positive integer and $f(x) := x^{2^n}+1$. In this paper, we study orders of primes dividing products of the form 
$P_{m,n}:=f(1)f(2)\cdots f(m)$. We prove that if $m > \max\{10^{12},4^{n+1}\}$, then there exists a prime divisor $p$ of $P_{m,n}$ such that 
$\ord_{p}(P_{m,n} )\leq n\cdot 2^{n-1}$. For $n=2$, we establish that for every positive integer $m$, there exists a prime divisor $p$ of 
$P_{m,2}$ such that $\ord_{p} (P_{m,2}) \leq 4$. Consequently, $P_{m,2}$ is never a fifth or higher power. This
extends work of Cilleruelo \cite{jc} who studied the case $n=1$.

\section{Introduction and main result}
For a prime $p$ and a nonzero integer $s$, define $\ord_{p}(s)$ to be the unique non-negative integer $i$ for which $p^i| s$ but $p^{i+1} \nmid s$. Let 
$f(x) \in  \mathbb{Z}[x]$ be a polynomial of degree $k \geq 2$ with positive leading coefficient which does not vanish at any positive integer. Set
$$
A_f(m):=f(1)f(2)\cdots f(m)
$$
and note that this is nonzero for all $m\in \mathbb{N}$ by the above assumption.

A major unsolved problem in analytic number theory concerns the question whether $f$ represents infinitely many primes if $f$ is irreducible and there exists no prime $p$ dividing $f(m)$ for all integers $m$. If $f$ represents infinitely many primes, then trivially, for infinitely many integers $m$, there exists a prime such that
$\ord_p(A_f(m))=1$. For particular polynomials $f$, several authors investigated the related question whether for all {\it sufficiently large} integers $m$ there exists a prime $p$ with $\ord_{p}(A_f(m)) = 1$. If this is the case, then, in particular, $A_f(m)$ is a perfect power for at most finitely many $m\in \mathbb{N}$.

Below we summarize a number of results from the literature. For the polynomial $f(x) = x^2 + 1$, J. Cilleruelo \cite{jc} proved the following result which we shall generalize in this paper.

\begin{theorem}[Cilleruelo] \label{Cill}
Let $f(x)=x^2+1$ and $m>3$. Then there exists a prime divisor $p$ of $A_f(m)$ with $\ord_{p}(A_f(m)) =1$. 
Consequently, $A_f(m)$ is a perfect power only for $m=3$, in which case we have $A_f(3)=10^2$.
\end{theorem}

Cilleruelo's work \cite{jc} used only elementary tools such as Chebyshev's upper bound inequality for the primes counting function. In subsequent work by Fang \cite{jhf}, his technique was applied to products $A_f(m)$ corresponding to the polynomials $4x^2+1$ and $2x^2-2x+1$. Yang,  Togb\'{e} and He \cite{yth} proved that for any irreducible quadratic polynomial $f(x) \in \mathbb{Z}[x]$, there exists a prime $p$ with $\ord_{p} (A_f(m)) = 1$ if $m \geq C$, where $C$ is a computable constant depending  on the coefficients of $f(x)$.

Furthermore, the above problem has been investigated by many authors for polynomials of the form $f(x) = x^k + 1$, where $k$ is an odd positive integer. G{\"u}rel and Kisisel \cite{gk} settled the case when $k = 3$. Based on an idea due to W. Zudilin, Zhang and Wang \cite{zw} extended this result to odd primes $k \geq 5$. Recently, Chen et al. \cite{cgr} managed to handle all odd prime powers $k$. Chen and Gong \cite{cg} treated the case when $k$ is a product of at most two odd primes and Dey and Laishram \cite{dl} managed to cover all $k$'s which are composed of at most four odd primes.

Thus, for polynomials $f(x) = x^k+1$ with $k$ odd, a lot of research has been done. With regard to even $k$'s, the authors are aware only of Cilleruelo's Theorem, stated above, for the case $k=2$. In this paper, we investigate orders of primes dividing $A_f(m)$ for polynomials of the form $f(x)=x^k+1$ when $k$ is a power of 2. Note that $f(x)=x^k+1$ is irreducible iff $k$ is a power of $2$. Throughout the sequel, we set
$$P_{m,n} := \prod\limits_{x\le m} \left(x^{2^n}+1\right).$$
We shall extend Cilleruelo's Theorem for the case $n=1$ to larger $n$'s as follows.

\begin{theorem}\label{1} 
Let $n\ge 2$ be an integer.  Then there exists a prime divisor $p$ of $P_{m,n}$ with $\ord_{p}(P_{m,n}) \leq n \cdot 2^{n-1}$ if
$m>\max\{10^{12},4^{n+1}\}$. Consequently, in this case, $P_{m,n}$ is not a perfect $q$-th power if $q$ is a positive integer exceeding $n\cdot 2^{n-1}$.
\end{theorem}

For $n=2$, we are able to remove the condition that $m>\max\{10^{12},4^{n+1}\}$, thus obtaining the following. 

\begin{theorem}\label{2}
For all positive integers $m$, there exists a prime divisor $p$ of $P_{m,2}$ with $\ord_{p}(P_{m,2}) \leq 4$. Consequently, 
$P_{m,2}$ is never a perfect $q$-th power if $q$ is a positive integer exceeding $4$. 
\end{theorem}

\section{Notations and preliminaries}
In this section, we provide some inequalities related to primes counting functions which are essential to prove our main results. As usual, we 
reserve the symbol "$p$" for primes and use the notations
$$
\pi(x):=\sum\limits_{p\le x} 1, \quad \pi(x;q,a):=\sum\limits_{\substack{p\le x\\ p\equiv a \bmod{q}}} 1, \quad 
\mbox{and} \quad  
\theta(x;q,a):=\sum\limits_{\substack{p\le x\\ p\equiv a \bmod{q}}} \log p
$$
throughout this paper.  Below are the lemmas that we shall use.

\begin{lemma} \label{prime0}
For any $x\ge 10^6$, we have
$$
\pi(x)\le 1.1\cdot \frac{x}{\log x}.
$$
\end{lemma}

\begin{proof}
In \cite{pd}, it was established that
\begin{equation*}
 \pi(x)\le \left(1+ \frac{1.2762}{\log x}\right) \cdot \frac{x}{\log x} \quad \mbox{if } x > 1. 
\end{equation*}
If $x \geq 10^6$, the desired bound follows. 
\end{proof}

\begin{lemma}\label{prime1}
For any integer $n\ge 2$ and any real $x\ge 4^{n+1}$, we have
$$
\pi(x;2^{n+1},1)\le \frac{4\cdot x}{2^n\log x}.
$$
\end{lemma}

\begin{proof}
The Brun-Titchmarsh inequality, as given by Montgomery and Vaughan \cite{MV}, asserts that
$$
\pi(x;q,a)\le \frac{2x}{\varphi(q)\log(x/q)}
$$
whenever $q<x$. This implies
$$
\pi(x;2^{n+1},1)\le \frac{2x}{2^n\log(x/2^{n+1})}.
$$
It follows that
$$
\pi(x;2^{n+1},1)\le \frac{2x}{2^n\log x}\cdot \frac{\log x}{\log(x/2^{n+1})}\\ \le \frac{2x}{2^n\log x}\cdot 
\frac{\log 4^{n+1}}{\log 2^{n+1}}\le \frac{4x}{2^n\log x}
$$
if $x\ge 4^{n+1}$, which completes the proof.
\end{proof} 

\begin{lemma}\label{sum3}
For any $x \ge 10^6$ and $a\in \{1,3,5,7\}$, we have
$$\sum\limits_{\substack {p \leq x \\ p \equiv a \bmod 8}} \frac{\log p}{p} > 0.245 \log x-3.15.$$
\end{lemma}

\begin{proof}
Using partial summation, we transform the sum in question into
\begin{equation} \label{parsum}
\begin{split}
\sum\limits_{\substack{p\le x\\ p\equiv a\bmod{8}}} \frac{\log p}{p} =  \frac{\theta(x;8,a)}{x}+\int\limits_{2}^x \frac{\theta(t;8,a)}{t^2} \mbox{d} t 
>  \frac{\theta(x;8,a)}{x}+\int\limits_{10^6}^x \frac{\theta(t;8,a)}{t^2} \mbox{d} t.
\end{split}
\end{equation}
By Corollary 1.7. in \cite{bmor}, we have
\begin{equation} \label{Coro}
\left|\theta(t;q,a)-\frac{t}{\varphi(q)}\right|< 0.024 \cdot \frac{t}{\log t} \quad \mbox{if } 1\le q\le 10^5,\ (q,a)=1 \mbox{ and } t\ge 10^6. 
\end{equation}
This implies
$$
\theta(t;8,a)>\frac{t}{4}-0.024 \cdot \frac{t}{\log t}
$$
if $t\ge 10^6$. Plugging this into \eqref{parsum}, and performing integration, we get
\begin{equation} \label{perform}
\sum\limits_{\substack{p\le x\\ p\equiv a\bmod{8}}} \frac{\log p}{p} > \frac{1}{4}-0.024\cdot \frac{1}{\log x}+\left[\frac{1}{4}\log t - 0.024\log\log t\right]_{10^6}^x.
\end{equation}
Since $\frac{\log x}{\log \log x}$ is an increasing function for $x \geq 10^6$, we have
$$\log \log x < 0.191\log x.$$
Using this inequality, we obain the desired result from \eqref{perform}.
\end{proof}

\section{Systems of congruences modulo prime powers} \label{conprime}
Let $n\in \mathbb{N}$ be given. A significant part of our method consists in finding an as small as possible number $N=N(n)$ such that for 
every partition 
\begin{equation} \label{part}
N=k_1+k_2+\cdots +k_s \quad (k_1\ge k_2\ge ...\ge k_s)
\end{equation}
of $N$ and any distinct $x_1,...,x_s \in \mathbb{N}$ satisfying a system of congruences of the form
\begin{equation} \label{system}
\begin{split}
x_1^{2^n}+1 \equiv & 0 \bmod{p^{k_1}}\\ 
\vdots\\ 
x_s^{2^n}+1 \equiv & 0 \bmod{p^{k_s}}
\end{split}
\end{equation}
with  $p$ an odd prime, it follows that
$$
p \ll_n x,
$$
where we set
$$
x:=\max\{|x_1|,...,|x_s|\}.
$$
It will become clear in section 8 how this problem, which is also of independent interest, enters the proof of Theorems \ref{1} and \ref{2}. The 
following sections 4 to 7 
are dedicated to solving this problem. The result will lead us directly to the quantity $n\cdot 2^{n-1}$ in Theorem \ref{1}.   

\section{Reformulation in Cyclotomic fields}
For a number field $K$, we denote by $\mathcal{O}_K$ the ring of algebraic integers in $K$. If $L$ is a finite extension of the number field $K$ and
$a\in L$, we denote by $N_{L:K}(a)$ the norm of $a$ over $K$. 

We write 
$$
x_i^{2^n}+1 = \prod\limits_{j=1}^{2^n} \left(x_i+\alpha_j\right),
$$
where $\left\{\alpha_1,...,\alpha_{2^n}\right\}$ is the set of primitive $2^{n+1}$-th roots of unity. Then
$$
p|(x_1^{2^n}+1) \Rightarrow \mathfrak{P}|(x_1+\alpha_1) \mathcal{O}_{\mathbb{Q}(\alpha_1)}
$$
for some prime ideal $\mathfrak{P}$ in $\mathcal{O}_{\mathbb{Q}(\alpha_1)}$ lying over $p$, but 
$$
\tilde{\mathfrak{P}} \nmid (x_1+\alpha_1)\mathcal{O}_{\mathbb{Q}(\alpha_1)}
$$ 
for any
prime ideal $\tilde{\mathfrak{P}}\not=\mathfrak{P}$ conjugate to $\mathfrak{P}$. Otherwise,  $\tilde{\mathfrak{P}}$ would divide both the ideals $(x_1+\alpha_1)\mathcal{O}_{\mathbb{Q}(\alpha_1)}$ and $(x_1+\alpha_j)\mathcal{O}_{\mathbb{Q}(\alpha_1)}$ for some $j\not=1$ and hence the ideal $(\alpha_1-\alpha_j)\mathcal{O}_{\mathbb{Q}(\alpha_1)}$. However, this is not possible because $p>2$ and the
discriminant
$$
\mbox{disc}(\mathcal{O}_{\mathbb{Q}(\alpha_1)})=\prod\limits_{1\le j_1<j_2\le 2^n} (\alpha_{j_1}-\alpha_{j_2})^2
$$
of $\mathcal{O}_{\mathbb{Q}(\alpha_1)}$ has no rational prime divisors other than 2. 
Hence, from the first congruence in \eqref{system}, it follows that
$$
\mathfrak{P}^{k_1}|(x_1+\alpha_1)\mathcal{O}_{\mathbb{Q}(\alpha_1)}.
$$
If $i\in \{2,...,s\}$, then
$\mathfrak{P}|(x_i+\alpha_{j_i})\mathcal{O}_{\mathbb{Q}(\alpha_1)}$ for some unique $j_i\in \{1,...,2^n\}$, and we have 
$$
\mathfrak{P}^{k_i}|(x_i+\alpha_{j_i})\mathcal{O}_{\mathbb{Q}(\alpha_1)}
$$
from the $i-th$ congruence in \eqref{system} by a similar argument as above. Set $j_1:=1$. Since $k_1\ge k_2\ge ...\ge k_s$, it follows that for every $r\in \{1,...,s\}$, we have
\begin{equation} \label{divides}
\mathfrak{P}^{k_r}|(x_i+\alpha_{j_i})\mathcal{O}_{\mathbb{Q}(\alpha_1)} \quad \mbox{if } 1\le i\le r.
\end{equation}

Now let $m< n$ be a non-negative integer and $r\in \{1,...,s\}$. Denote by $\zeta_k$ a primive $k$-th root of unity. If  
$$
\beta_1,...,\beta_r \in \mathcal{O}_{\mathbb{Q}(\zeta_{2^{m+1}})}
$$
are such that
$$
\beta_1\alpha_{j_1}+\cdots +\beta_r\alpha_{j_r} = 0,
$$
then
$$
\beta_1(x_1+\alpha_{j_1})+\cdots +\beta_r(x_r+\alpha_{j_r}) \in \mathcal{O}_{\mathbb{Q}(\zeta_{2^{m+1}})}.
$$
Since also
$$
\mathfrak{P}^{k_r}| (\beta_1(x_1+\alpha_{j_1})+\cdots +\beta_r(x_r+\alpha_{j_r}))\mathcal{O}_\mathbb{Q}(\zeta_{2^{n+1}})
$$
using \eqref{divides}, it follows that
$$
\left(\mathfrak{P} \cap \mathcal{O}_\mathbb{Q}(\zeta_{2^{m+1}})\right)^{k_r}| (\beta_1(x_1+\alpha_{j_1})+\cdots +\beta_r(x_r+\alpha_{j_r}))\mathcal{O}_\mathbb{Q}(\zeta_{2^{m+1}})
$$
and hence
$$
p^{k_r} | N_{\mathbb{Q}(\zeta_{2^{m+1}}):\mathbb{Q}}\left(\beta_1(x_1+\alpha_{j_1})+\cdots +\beta_r(x_r+\alpha_{j_r})\right).
$$
If in addition 
$$
\beta_1x_1+\cdots +\beta_rx_r\not=0,
$$
then we deduce that
\begin{equation} \label{keybound}
p^{k_r}\le N_{\mathbb{Q}(\zeta_{2^{m+1}}):\mathbb{Q}}\left(\beta_1(x_1+\alpha_{j_1})+\cdots +\beta_r(x_r+\alpha_{j_r})\right)
\end{equation}
and hence
$$
p^{k_r} \ll_{n,m,\beta_1,...,\beta_r} x^{2^m},
$$
which implies
$$
p \ll_{n,m,\beta_1,...,\beta_r} x,
$$
provided that
\begin{equation} \label{krcond}
k_r\ge 2^m.
\end{equation}
Clearly, we also have the bound 
$$
p \ll_n x
$$
if
\begin{equation} \label{krcondi}
k_r\ge 2^n.
\end{equation}

Now we consider an arbitrary non-negative integer $m$. If $m\ge n$, then we set $R(m,n):=1$. If $m<n$, then let $R(m,n)$ be the smallest number $r$ such that given any $r$ primitive $2^{n+1}$-th roots of unity
$\gamma_1,...,\gamma_r$ (not necessarily distinct), there exist 
$$
\beta_1,...,\beta_r\in \mathcal{O}_{\mathbb{Q}(\zeta_{2^{m+1}})}
$$
such that
$$
\beta_1\gamma_1+\cdots + \beta_r\gamma_r=0 
$$
and 
$$
\beta_1x_1+\cdots + \beta_rx_r\not=0
$$
for any distinct positive integers $x_1,...,x_r$. Then it follows that
$$
p\ll_n x,
$$
provided that for any partition of the form in \eqref{part}
we have
\begin{equation}
r\ge R(m,n) \quad \mbox{and} \quad
k_r\ge 2^m
\end{equation}
for some $r\in \{1,...,s\}$ and $m\in \{1,...,n\}$. 

Note that $R(m,n)$ decreases as $m$ increases. Hence, we may choose
$$
m:=\lfloor \log_2 k_r \rfloor
$$
and our above condition reduces to 
\begin{equation} 
r\ge R(\lfloor \log_2 k_r \rfloor,n) \quad \mbox{for some } r\in \{1,...,s\}.
\end{equation}

\section{Determining $R(m,n)$}
Throughout the sequel, for any real number $x$, we denote by $\lceil x \rceil$ the smallest integer greater or equal $x$ and by $\lfloor x \rfloor$ the largest integer less or equal $x$. We now prove the following. 

\begin{lemma} \label{Rformula} For any natural numbers $m$ and $n$, we have 
$$
R(m,n)= \lfloor 2^{n-m-1}\rfloor +1.
$$
\end{lemma}

\begin{proof}
This is trivial if $m\ge n$. So assume $m<n$. Then we claim that among $2^{n-m-1}+1$ (not necessarily distinct) primitive $2^{n+1}$-th roots of unity, there exist two, 
$\gamma_1$ and $\gamma_2$, such that $\beta:=\gamma_1/\gamma_2$ is a $2^{m+1}$-th root of unity. This is equivalent to saying that among $2^{n-m-1}+1$ (not necessarily distinct) 
odd integers in $\{1,3,...,2^{n+1}-1\}$, there exist two whose difference is divisible by $2^{n-m}$. Indeed, these integers fall into 
$2^{n-m-1}$ possible residue classes modulo $2^{n-m}$.
By pigeonhole principle, two of them fall into the same residue class, and hence the claim follows. 

Now, for $\gamma_1$, $\gamma_2$, $\beta$ as above, we have $\gamma_1-\beta\gamma_2=0$, but clearly $x_1-\beta x_2\not=0$ for any two distinct positive integers $x_1$ and $x_2$. This proves that 
$$
R(m,n)\le  2^{n-m-1}+1.
$$ 

It remains to show that
\begin{equation} \label{second}
R(m,n)> 2^{n-m-1}.
\end{equation}
Assume the contrary. We look at the example
$$
\gamma_j=\zeta_{2^{n+1}}^{2j-1}, \quad j=1,...,2^{n-m-1}.
$$
These are primive $2^{n+1}$-th roots of unity. Moreover, the set
$$
\{\zeta_{2^{n+1}}^{-1}\gamma_1,...,\zeta_{2^{n+1}}^{-1}\gamma_{2^{n-m-1}}\}
$$
forms an integral basis of $\mathbb{Q}(\zeta_{2^{n}})$ over $\mathbb{Q}(\zeta_{2^{m+1}})$. Hence, if $\beta_1,...,\beta_{2^{n-m-1}}\in 
\mathcal{O}_{\mathbb{Q}(\zeta_{2^{m+1}})}$, then 
$$
\beta_1\gamma_1+...+\beta_{2^{n-m-1}}\gamma_{2^{n-m-1}}=0
$$
implies $\beta_1=...=\beta_{2^{n-m-1}}=0$. But then our second condition 
$$
\beta_1x_1+...+\beta_{2^{n-m-1}}x_{2^{n-m-1}}\not=0
$$
is violated for any integers $x_1,...,x_{2^{n-m-1}}$. This gives a contradiction. Hence, \eqref{second} follows which completes the proof. 
\end{proof}
{\bf Remark.} For our purposes, it would have been sufficient to prove that $R(m,n)\le  2^{n-m-1}+1$. 

\section{Transformation into a combinatorial condition}
By the considerations in the previous section, an admissible $N(n)$ is the smallest integer $N$ such that every partition of the form in \eqref{part} 
satisfies the condition
\begin{equation} \label{cond}
r\ge \lfloor 2^{n-\lfloor \log_2 k_r \rfloor -1} \rfloor +1 \quad \mbox{for some } r\in \{1,...,s\}.
\end{equation}
In the following, we construct an extreme partition satisfying \eqref{cond}. 
We take $k_1,...,k_{s}$ as large as possible such that \eqref{cond} is {\it not} satisfied if $1\le r\le s-1$ but satisfied if $r=s$. This property
determines the partition in question completely, namely we obtain
$$
s=2^{n-1}+1,
$$
$$
k_r=2^{n-\lceil \log_2 r \rceil}-1 \quad \mbox{if } r\le s-1
$$
and 
$$
k_s=1.
$$
Moreover, we calculate that 
\begin{equation}
\begin{split}
N= & k_1+\cdots +k_s\\ = & (2^n-1)+ \sum\limits_{0\le j\le n-2} 2^j\left(2^{n-1-j}-1\right)+1\\
= & 2^n+(n-1)\cdot 2^{n-1}-\sum\limits_{0\le j\le n-2} 2^j\\ 
= & n\cdot 2^{n-1}+1. 
\end{split}
\end{equation}

{\bf Examples:} For $n=1,2,3,4,5$, we get the partitions 
\begin{equation*}
\begin{split}
2=&1+1,\\
5=&3+1+1,\\
13=&7+3+1+1+1,\\
33=& 15+7+3+3+1+1+1+1+1,\\
81=& 31+15+7+7+3+3+3+3+1+1+1+1+1+1+1+1+1.
\end{split}
\end{equation*}

In the following, we prove that this actually gives exactly the minimal number $N$ we are aiming for. 

\begin{lemma}
The number
$$
N(n)=n\cdot 2^{n-1}+1
$$
is the smallest positive integer $N$ such that every partition of the form in \eqref{part} satisfies the condition \eqref{cond}.
\end{lemma}

\begin{proof}
Look at the extreme partition constructed above. $N$ cannot be chosen smaller because 
$$
n\cdot 2^{n-1}=k_1+\cdots + k_{s-1}
$$
is a partition of $n\cdot 2^{n-1}$ which does not satisfy the required condition. Now, if 
$$
N=n\cdot 2^{n-1}+1=k_1'+...+k_{s'}' \quad (k_1'\ge...\ge k_{s'}')
$$
is any partition different from our extreme partition above, then $s'\ge s$ or $k_r'>k_r$ for some $r\in \{1,...,s'\}$. In the first case,
$$
s'\ge s = 2^{n-1} +1 \ge \lfloor 2^{n-\lfloor \log_2 k_{s'}' \rfloor -1} \rfloor +1.
$$
In the second case, 
$$
r\ge \lfloor 2^{n-\lfloor \log_2 k_r' \rfloor -1} \rfloor +1 
$$
by construction of the partition $N=k_1+...+k_s$. Hence, our new partition $N=k_1'+...+k_{s'}'$ satisfies the desired condition. This completes the proof. 
\end{proof}

\section{Back to congruences modulo prime powers}
Now we are ready to prove what we formulated as a goal in section \ref{conprime}. In addition, we observe that we even get an upper bound for $p$ which does only depend on $x$ and not on $n$. 
Indeed, taking our proof of Lemma \ref{Rformula} in consideration, we may choose $\beta_u=1$ and $\beta_v=-\beta$ with $\beta$ a $2^{m+1}$-th root of unity for suitable distinct $u,v\in \{1,...,r\}$ and
$\beta_i=0$ if $i\not\in \{1,...,r\}\setminus \{u,v\}$. This implies 
$$
|\sigma(\beta_1)(x_1+\sigma(\alpha_{j_1}))+\cdots +\sigma(\beta_r)(x_r+\sigma(\alpha_{j_r}))|\le 2(x+1)
$$
for any $\sigma\in$ Gal$(\mathbb{Q}(\zeta_{2^{n+1}}):\mathbb{Q})$ and hence 
$$
N_{\mathbb{Q}(\zeta_{2^{n+1}}):\mathbb{Q}}\left(\beta_1(x_1+\alpha_{j_1})+\cdots +\beta_r(x_r+\alpha_{j_r})\right)\le 2^{2^n}(x+1)^{2^n}.
$$
Upon recalling that $\beta_1(x_1+\alpha_{j_1})+\cdots +\beta_r(x_r+\alpha_{j_r})\in \mathcal{O}_{\mathbb{Q}(\zeta_{2^{m+1}})}$, it 
follows that
$$
N_{\mathbb{Q}(\zeta_{2^{m+1}}):\mathbb{Q}}\left(\beta_1(x_1+\alpha_{j_1})+\cdots +\beta_r(x_r+\alpha_{j_r})\right)\le 2^{2^m}(x+1)^{2^m}.
$$
Under the condition \eqref{krcond} which says that $k_r\ge 2^m$, it follows now from \eqref{keybound} that 
$$
p\le 2(x+1).
$$
Similarly, we find that
$$
p\le x+1
$$
under the condition \eqref{krcondi} which says that $k_r\ge 2^n$. 

Summarizing our results in the previous sections and taking our observation above into account, we thus have established the following.

\begin{theorem}\label{3}
Let $n\in \mathbb{N}$ be given. Set $N=N(n):=n\cdot 2^{n-1}+1$. Then if  
$$
N=k_1+k_2+\cdots +k_s \quad (k_1\ge k_2\ge ...\ge k_s)
$$
is any partition of $N$ and the system of congruences 
\begin{equation*}
\begin{split}
x_1^{2^n}+1 \equiv & 0 \bmod{p^{k_1}}\\ 
\vdots\\ 
x_s^{2^n}+1 \equiv & 0 \bmod{p^{k_s}}
\end{split}
\end{equation*}
holds for distinct $x_1,...,x_s \in \mathbb{N}$ and $p$ an odd prime, then  
$$
p \le 2(x+1),
$$
where $x:=\{|x_1|,...,|x_s|\}$.
\end{theorem}

\section{proof of Theorem \ref{1}}
Let us assume that $n\ge 2$, $m>\max\{10^{12},4^{n+1}\}$ and 
$$\ord_{p}(P_{m,n}) > n \cdot 2^{n-1}$$
for all primes $p$ dividing $P_{m,n}$. Then from Theorem \ref{3}, it follows that
\begin{equation} \label{prange}
p \le 2(m+1)
\end{equation}
for all these primes. 
(This will be essential in our proof.)
Hence, we can write $P_{m,n}$ as
\begin{equation}\label{eq1}
P_{m,n} = \prod\limits_{p \le 2(m+1)} p^{{\alpha}_p},
\end{equation}
where the ${\alpha}_p$'s are non-negative integers with either ${\alpha}_p = 0$ or ${\alpha}_p > n \cdot 2^{n-1}$. Clearly,
\begin{equation}\label{eq2}
P_{m,n} > \prod\limits_{x\le m} x^{2^n}= (m!)^{2^n}.
\end{equation}
Write
\begin{equation}\label{eq3}
m! = \prod\limits_{p \leq m} p^{{\beta}_p},
\end{equation}
where the ${\beta}_p$'s are positive integers.

Combining \eqref{eq1}, \eqref{eq2} and \eqref{eq3}, we have
$$\Big(\prod\limits_{p \leq m} p^{{\beta}_p}\Big)^{2^n} < \Big(\prod\limits_{p \le 2(m+1)} p^{{\alpha}_p}\Big).$$
Taking logarithm, if follows that
\begin{equation}\label{eq4}
\sum\limits_{p \leq m} {\beta}_p \log p < \frac{1}{2^n} \sum\limits_{p \le 2(m+1)} {\alpha}_p \log p.
\end{equation}
Since
$$
x^{2^n} + 1 \equiv
\left\{\begin{array}{cl} 
 1 \bmod 4 & \mbox{ if } x  \mbox{ is even, } \\
 2 \bmod 4 & \mbox{ if } x  \mbox{ is odd, }
\end{array}\right.
$$ 
we see that
\begin{equation}\label{eq5}
{\alpha}_2 = \Big\lceil \frac{m}{2} \Big\rceil.
\end{equation}

Now let $p$ be an odd prime dividing $x^{2^n}+1$. Then $p \equiv 1 \bmod {2^{n+1}}$ and, moreover, there are exactly $2^n$ solutions to the congruence 
$$x^{2^n}+1 \equiv 0 \bmod p.$$
(Note that the $2^{n+1}$-th cyclotomic field $\mathbb{Q}(\zeta_{2^{n+1}})$ is the splitting field of the polynomial $x^{2^n}+1$ over $\mathbb{Q}$, and the
rational primes which split completely in this field are exactly those congruent to $1 \bmod{2^{n+1}}$.)  
By Hensel's lemma, they extend uniquely to solutions of 
$$x^{2^n}+1 \equiv 0 \bmod {p^j},$$
for any $j \geq 1$. Thus, each interval of length $p^j$ contains exactly $2^n$ solutions of this congruence. It follows that
\begin{equation}\label{eq6}
\begin{split}
{\alpha}_p \quad = \sum\limits_{j \leq \frac{\log(m^{2^n}+1)}{\log p}} \sharp\{x \leq m\ :\ p^j | (x^{2^n}+1)\}
  \leq \sum\limits_{j \leq \frac{\log(m^{2^n}+1)}{\log p}} 2^n \Big\lceil \frac{m}{p^j}\Big\rceil.
\end{split}
\end{equation}
Also we have
\begin{equation}\label{eq7}
\begin{split}
{\beta}_p \quad = \sum\limits_{j \leq \frac{\log m}{\log p}} \sharp\{x \leq m\ :\ p^j \mid x\}
  = \sum\limits_{j \leq \frac{\log m}{\log p}} \Big\lfloor \frac{m}{p^j} \Big\rfloor.
\end{split}
\end{equation}

From \eqref{eq6} and \eqref{eq7}, we deduce that
\begin{equation}\label{eq8}
\begin{split}
\frac{{\alpha}_p}{2^n} - {\beta}_p \quad & \leq \sum\limits_{j \leq \frac{\log(m^{2^n}+1)}{\log p}} \Big\lceil \frac{m}{p^j}\Big\rceil - \sum\limits_{j \leq \frac{\log m}{\log p}} \Big\lfloor \frac{m}{p^j} \Big\rfloor\\
 = &\sum\limits_{j \leq \frac{\log m}{\log p}} \Big(\Big\lceil \frac{m}{p^j} \Big\rceil - \Big\lfloor \frac{m}{p^j} \Big\rfloor \Big)+ \sum\limits_{\frac{\log m}{\log p} < j \leq \frac{\log(m^{2^n}+1)}{\log p}} \Big\lceil \frac{m}{p^j} \Big\rceil \\
  & \leq \sum\limits_{j \leq \frac{\log m}{\log p}} 1 + \sum\limits_{\frac{\log m}{\log p} < j \leq \frac{\log(m^{2^n}+1)}{\log p}} 1\\
  & \le \frac{\log(m^{2^n}+1)}{\log p}. 
\end{split}
\end{equation}
Combining inequalities \eqref{eq4}, \eqref{eq5} and \eqref{eq8}, and recalling that $\alpha_p=0$ if $p>2$ and $p\not\equiv 1 \bmod{2^{n+1}}$, we obtain
\begin{equation}\label{eq9}
\begin{split}
& \sum\limits_{\substack {p \leq m \\ p \not\equiv 1 \bmod {2^{n+1}}}} {\beta}_p \log p\\  <  & - \sum\limits_{\substack {p \leq m \\ p \equiv 1 \bmod {2^{n+1}}}} {\beta}_p \log p + \frac{1}{2^n} \sum\limits_{\substack {p \leq {2(m+1)} \\ p \equiv 1 \bmod {2^{n+1}}}} {\alpha}_p \log p + \frac{1}{2^n} {\alpha}_2 \log 2\\
= & \frac{1}{2^n} \Big\lceil \frac{m}{2} \Big\rceil \log 2 + \sum\limits_{\substack {p \leq m \\ p \equiv 1 \bmod {2^{n+1}}}} \Big(\frac{{\alpha}_p}{2^n} - {\beta}_p \Big) \log p + \frac{1}{2^n} \sum\limits_{\substack{m <p \leq 2(m+1)\\ p \equiv 1 \bmod {2^{n+1}}}} {\alpha}_p \log p\\
<  &\frac{1}{2^n} \Big\lceil \frac{m}{2} \Big\rceil \log 2 + \pi (m;2^{n+1},1) \log (m^{2^n}+1) + \frac{1}{2^n} 
\sum\limits_{\substack{m <p \leq 2(m+1)\\ p \equiv 1 \bmod {2^{n+1}}}} {\alpha}_p \log p.
\end{split} 
\end{equation}

If $p > m$, then from \eqref{eq6}, we have
\begin{equation}\label{eq10}
{\alpha}_p < 2^{2n}
\end{equation}
since
$$\log(m^{2^n}+1) < 2^n \log(m+1) \leq 2^n \log p.$$
Moreover, from \eqref{eq7}, we have
$${\beta}_p \geq \sum\limits_{j \leq M} \Big\lfloor \frac{m}{p^j} \Big\rfloor \ge \frac{m(1-p^{-M})}{p-1} - M,$$
where $M = \lfloor \frac{\log m}{\log p} \rfloor$.
Hence, for $p \leq m$, we deduce that
\begin{equation}\label{eq11}
{\beta}_p \geq \frac{m-p}{p-1} -M \geq \frac{m-1}{p-1} - \frac{2\log m}{\log p}.
\end{equation}

Combining inequalities \eqref{eq9}, \eqref{eq10} and \eqref{eq11}, we get
\begin{equation*}
\begin{split}
& \sum\limits_{\substack {p \leq m \\ p \not\equiv 1 \bmod {2^{n+1}}}} \Big(\frac{m-1}{p-1} - \frac{2\log m}{\log p}\Big) \log p\\ 
< & \frac{1}{2^n} \Big\lceil \frac{m}{2} \Big\rceil \log 2 + \pi (m;2^{n+1},1) \log (m^{2^n}+1)+   2^n \sum\limits_{\substack{m <p \le 2(m+1)\\  p \equiv 1 \bmod {2^{n+1}}}} \log p,
\end{split}
\end{equation*}
which implies
\begin{equation}\label{eq13}
\begin{split}
& (m-1)\sum\limits_{\substack {p \leq m \\ p \not\equiv 1 \bmod {2^{n+1}}}} \frac{\log p}{p-1}\\ 
< & 2\pi(m)\log m+\frac{m+1}{2^{n+1}} \log 2 + \pi (m;2^{n+1},1) \log (m^{2^n}+1)\\ & + 2^n \left(\pi(2(m+1);2^{n+1},1)-\pi(m;2^{n+1},1)\right)\log(2(m+1))\\
< & 2\pi(m)\log m+\frac{m+1}{2^{n+1}} \log 2 + 2^n \pi(2(m+1);2^{n+1},1) \log(2(m+1)).
\end{split}
\end{equation}
Now recalling that $n\ge 2$ and using Lemma \ref{sum3}, we have
\begin{equation}\label{eq14}
\sum\limits_{\substack{p \leq m \\ p \not\equiv 1 \bmod {2^{n+1}}}} \frac{\log p}{p-1} \geq 
\sum\limits_{a\in \{3,5,7\}} 
\sum\limits_{\substack{p \leq m \\ p \equiv a \bmod 8}} \frac{\log p}{p}>3 (0.245\log m-3.15).
\end{equation}
Combining inequalities \eqref{eq13} and \eqref{eq14}, applying Lemmas \ref{prime1}, and dividing by $m-1$, we obtain
\begin{equation}\label{eq16}
\begin{split}
3 (0.245\log m-3.15) < & \frac{2.2\cdot m}{m-1}+\frac{m+1}{m-1}\cdot \frac{\log 2}{2^{n+1}}  + \frac{8(m+1)}{m-1}.
\end{split} 
\end{equation}
Note that the limit, as $m\rightarrow \infty$, of the right-hand side is 
$$
(\log 2)\cdot 2^{-(n+1)}+10.2.
$$ 
Hence, if $m$ is large enough, then the above inequality will be
false. An easy calculation shows that this is the case whenever $m>10^{12}$ and hence we reach a contradiction. We conclude that
there exists a prime $p$ with $\ord_p(P_{m,n})\le n \cdot 2^{n-1}$, which completes the proof.   

{\bf Remark:} We note that the above argument would not go through if we had some much weaker condition like $p<\beta n^{\alpha}$ with $\alpha>1$ in place of \eqref{prange}.

\section{Proof of Theorem \ref{2}} 
By Theorem \ref{1}, there exists a prime $p$ with $\ord_{p} P_{m,2}\le 4$ if $m>10^{12}$. Therefore, it suffices to check that the same holds if $1 \leq m \leq 10^{12}$. The claim is trivial if $1 \leq m \leq 5$ since $1 \leq \ord_2(P_{m,2}) \leq 3$. Further, we observe that $6^4 + 1 = 1297$ is a prime and the next $x$'s for which $1297$ divides $x^4 + 1$ are $x = 216, 1081, 1291, 1303$. Moreover, for these $x$'s, 
$\ord_{1297}(x^4+1)=1$. Hence, there exists a prime $p$ with $\ord_{p}(P_{m,2}) \leq 4$ if $6 \leq m \leq 1302$. Next, we observe that $1302^4 + 1 = 2873716601617$ is a prime as well. The next $x$'s for which $2873716601617$ divides $x^4 + 1$ are $x = 2207155608, 2871509446009, 2873716600315, 2873716602919$. Moreover, for these $x$'s, 
$\ord_{2873716601617}(x^4+1)=1$. Hence, we conclude that there exists a prime $p$ with $\ord_{p}(P_{m,2}) \leq 4$ if $1302 \leq m \leq 2873716602918$. This completes the proof.


\begin{thebibliography}{3}
\bibitem{ta}
T.M. Apostol, {\it Introduction to Analytic Number Theory}, Springer-Verlag, New York, 1976.

\bibitem{bmor}
M.A. Bennett, G. Martin, K. Obryant and A. Rechnitzer, Explicit bounds for primes in arithmetic progressions,  
{Illinois J. Math.} {\bf 62} (2018), 427-532.

\bibitem{cg}
Y.-G. Chen and M.-L. Gong, On the products $(1^l+1)(2^l+1)\cdots(n^l+1)$ II, {\it J. Number Theory} {\bf 144} (2014), 176-187.

\bibitem{cgr} 
Y.-G. Chen, M.-L. Gong and X.-Z. Ren, On the products $(1^l+1)(2^l+1)\cdots(n^l+1)$, {\it J. Number Theory} {\bf 133} (2013), 2470-2474.

\bibitem{pd}
P. Dusart, The $k^{th}$ th prime is greater than $k(\ln k + \ln\ln k - 1)$ for $k \geq 2$, {\it Math. Comp.} {\bf 68} (1999), 411-415.

\bibitem{jc}
J. Cilleruelo, Squares in $(1^2+1)(2^2+1)\cdots(n^2+1)$, {\it J. Number Theory} {\bf 128} (2008), 2488-2491. 

\bibitem{dl}
P.K. Dey and S. Laishram, Powerful numbers in the product of consecutive integer values of a polynomial, {\it Publ. Math. Debrecen} {\bf 94} (2019), no. 3-4, 319-336.

\bibitem{jhf} 
J.-H. Fang, Neither $\prod_{k=1}^{n}(4k^2+1)$ nor $\prod_{k=1}^{n}(2k(k-1)+1)$ is a perfect square, {\it Integers} {\bf 9} (2009), 177-180.

\bibitem{gk} 
E. G{\"u}rel and A.U.O. Kisisel, A note on the products $(1^{\mu}+1)(2^{\mu}+1)\cdots(n^{\mu}+1)$, {\it J. Number Theory} {\bf 130} (2010), 187-191.

\bibitem{hr}
G. Hardy and E. Wright, {\it An Introduction to the Theory of Numbers}, Oxford Univ. Press, 1980.

\bibitem{MV} H.L. Montgomery; R.C. Vaughan, {\it The large sieve}, Mathematika 20 (1973), 119–134.

\bibitem{yth}
S.-C. Yang, A. Togb\'{e} and B. He, Diophantine equations with products of consecutive values of a quadratic polynomial, {\it J. Number Theory} {\bf 131} (2011), 1840-1851.


\bibitem{zw} 
T. Wang and W. Zhang, Powerful numbers in $(1^k+1)(2^k+1)\cdots(n^k+1)$, {\it J. Number Theory} {\bf 132} (2012), 2630-2635.
 \end{thebibliography}
\end{document}